\documentclass[11pt,reqno]{amsart}

\usepackage[a4paper,margin=1.08in]{geometry}
\usepackage[T1]{fontenc}
\usepackage{lmodern}
\usepackage{microtype}
\usepackage{mathtools,amssymb,amsfonts}
\usepackage{booktabs,array}
\usepackage{enumitem}
\usepackage{xcolor}
\usepackage[colorlinks=true,linkcolor=blue!48!black,citecolor=green!35!black,urlcolor=blue!48!black]{hyperref}
\usepackage{aliascnt}
\usepackage[nameinlink,capitalise,noabbrev]{cleveref}
\usepackage{setspace}

\hypersetup{
	pdftitle={Davenport Constants of Finite Abelian Groups},
	pdfauthor={Dongchun Han and Hanbin Zhang},
	pdfsubject={Davenport constants of finite abelian groups},
	pdfkeywords={Davenport constant, zero-sum sequences, finite abelian groups, group algebras}
}

\allowdisplaybreaks
\setlist{itemsep=0.18em,topsep=0.28em}

\numberwithin{equation}{section}

\newtheorem{theorem}{Theorem}[section]
\newaliascnt{proposition}{theorem}
\newtheorem{proposition}[proposition]{Proposition}
\aliascntresetthe{proposition}
\newaliascnt{lemma}{theorem}
\newtheorem{lemma}[lemma]{Lemma}
\aliascntresetthe{lemma}
\newaliascnt{corollary}{theorem}

\aliascntresetthe{corollary}
\newaliascnt{claim}{theorem}

\aliascntresetthe{claim}
\theoremstyle{definition}
\newaliascnt{definition}{theorem}

\aliascntresetthe{definition}
\newaliascnt{remark}{theorem}

\aliascntresetthe{remark}
\newaliascnt{question}{theorem}

\aliascntresetthe{question}

\crefname{theorem}{Theorem}{Theorems}
\Crefname{theorem}{Theorem}{Theorems}
\crefname{proposition}{Proposition}{Propositions}
\Crefname{proposition}{Proposition}{Propositions}
\crefname{lemma}{Lemma}{Lemmas}
\Crefname{lemma}{Lemma}{Lemmas}
\crefname{corollary}{Corollary}{Corollaries}
\Crefname{corollary}{Corollary}{Corollaries}
\crefname{claim}{Claim}{Claims}
\Crefname{claim}{Claim}{Claims}
\crefname{definition}{Definition}{Definitions}
\Crefname{definition}{Definition}{Definitions}
\crefname{remark}{Remark}{Remarks}
\Crefname{remark}{Remark}{Remarks}
\crefname{question}{Question}{Questions}
\Crefname{question}{Question}{Questions}

\newcommand{\D}{\mathsf D}
\newcommand{\Ds}{\mathsf D^{*}}
\newcommand{\F}{\mathbb F}

\newcommand{\rank}{\operatorname{rank}}
\newcommand{\expn}{\operatorname{exp}}
\newcommand{\Hom}{\operatorname{Hom}}
\newcommand{\KL}{\mathcal D}

\newcommand{\Def}{\Delta}

\newcommand{\bdot}{\boldsymbol{\cdot}}

\title[Davenport constants of finite abelian groups]{Davenport Constants of Finite Abelian Groups}
\author{Dongchun Han}
\address{School of Mathematics, Southwest Jiaotong University, Chengdu 610000, P.R. China}
\email{handongchun@swjtu.edu.cn}

\author{Hanbin Zhang}
\address{School of Mathematics (Zhuhai), Sun Yat-sen University, Zhuhai 519082, Guangdong, P.R. China}
\email{zhanghb68@mail.sysu.edu.cn}


\begin{document}
	\raggedbottom

\begin{abstract}
A central problem in zero-sum theory is to determine the
Davenport constant of a finite abelian group. For a finite abelian group
$G\cong C_{n_1}\oplus\cdots\oplus C_{n_r}$, with
$1<n_1\mid\cdots\mid n_r$, let $\D(G)$ be the least integer $t$ such that
every sequence of $t$ elements of $G$ contains a nonempty zero-sum
subsequence, and denote $\Ds(G)\coloneqq 1+\sum_{i=1}^r(n_i-1)$. The standard
construction shows that $\D(G)\ge\Ds(G)$. Equality holds for several important
classes of groups but fails in general, and even the question whether
$\D(G)/\Ds(G)$ is bounded by an absolute constant remains open. In this paper, we prove
the upper bound
\[
 \D(G)<4.86\,\expn(G)\frac{\log|G|}{\log\expn(G)}
\]
for every nontrivial finite abelian group. We also prove that, for almost all finite abelian groups (counting isomorphism classes
by order),
\[
 \D(G)\sim\expn(G)\frac{\log|G|}{\log\expn(G)}\sim\Ds(G).
\]
We also study the difference $\D(G)-\Ds(G)$: for every function $L$ with
$L(x)\to\infty$, almost all groups satisfy
$\D(G)-\Ds(G)\le L(\Ds(G))$. However, for every fixed $C\ge0$, the groups with $\D(G)-\Ds(G)>C$ have positive lower density.
\end{abstract}
\maketitle

\section{Introduction}\label{sec:introduction}

Throughout, $G$ denotes a nontrivial finite abelian group, written
additively. For a finite abelian group $G$, written additively, the Davenport constant
$\D(G)$ is the least positive integer $t$ such that every sequence of $t$
elements of $G$ contains a nonempty zero-sum subsequence. Write
$G\cong C_{n_1}\oplus\cdots\oplus C_{n_r}$, where
$1<n_1\mid\cdots\mid n_r$, and put $\expn(G)=n_r$. The standard
construction gives
\begin{equation}\label{eq:Dstar}
 \D(G)\ge\Ds(G):=1+\sum_{i=1}^r(n_i-1).
\end{equation}
Determining $\D(G)$ is a basic problem of zero-sum theory, originating in
questions of nonunique factorization in algebraic number fields; see
\cite{GaoGeroldinger2006,GeroldingerHalterKoch2006,NarkiewiczSliwa1982,Rogers1963}.
It is related to the Erd\H{o}s--Ginzburg--Ziv theorem \cite{EGZ} and has
connections with invariant theory
\cite{CziszterDomokosGeroldinger,HanZhang2019,Schmid1991} and coding theory
\cite{MarchanOSS2015,PlagneS2011}.

In the 1960s, $\D(G)=\Ds(G)$ was verified for $p$-groups and for groups
of rank at most two; see \cite{EmdeBoasKruyswijk1969,Olson1969,Olson1969II}, the historical accounts in
~\cite[Theorem 3.1]{GaoGeroldinger2006}, and
~\cite[Theorems 5.5.9 and 5.8.3]{GeroldingerHalterKoch2006}. In general, the exact values of $\D(G)$ are unknown. A standing conjecture is that
\(
 \D(G)=\Ds(G)
\) for $G=C_n^r$ or $\rank(G)=3$;
 see \cite[Conjecture 3.5]{GaoGeroldinger2006} (and the references therein), and we also refer to \cite{BisMaz,GirardSchmid2019,GYW2026} for some recent studies. On the other hand, the equality
$\D(G)=\Ds(G)$ fails in general, which was first noted in \cite{EmdeBoasKruyswijk1969}. Geroldinger and Schneider \cite{GeroldingerSchn} constructed infinite families of rank-four groups with $\D(G)>\Ds(G)$;
see also \cite{Mazur1992}. Liu
\cite[Corollary~4.7]{Liu2020} showed that $\D(G)-\Ds(G)$ can be
arbitrarily large at fixed non-prime-power exponent, and Wang
\cite[Theorem~3.1]{Wang2026} proved that it is unbounded at every fixed
rank $r\ge8$. Dimitrov \cite[Section~2.2]{Dimitrov2007} asked whether
$\D(G)/\Ds(G)$ is bounded by an absolute constant. His estimate
$\D(G)/\Ds(G)<(500r\log r)^r$ for $r\ge2$
\cite[Theorem~1]{Dimitrov2007} gives a bound at each fixed rank; a bound
independent of the rank remains open. For related results, see
\cite{BisMaz,Gao2000,GeroldingerGryn,GirardSchmid2019,Mazur1992}.

A classical bound of van Emde Boas and Kruyswijk
\cite[Theorem~7.1]{EmdeBoasKruyswijk1969} is
\begin{equation}\label{eq:classical}
 \D(G)\le \expn(G)\left(1+\log\frac{|G|}{\expn(G)}\right);
\end{equation}
see also \cite[Theorem~3.6(1)]{GaoGeroldinger2006} and
\cite[Theorem~1]{Meshulam1990}. This estimate was used in the proof of the
infinitude of Carmichael numbers \cite{AGP1994}. For other general and
fixed-rank estimates, see
\cite{AlonDubiner1995,BalasubramanianBhowmik2007,Girard2018}.
Bhowmik and Schlage-Puchta
\cite[Theorem~1.1]{BhowmikSchlagePuchta2012} proved that
$\D(G)\le\expn(G)+|G|/\expn(G)-1$ when
$\expn(G)\ge\sqrt{|G|}$. Wang
\cite[Theorem~6.2 and Remark~6.9]{Wang2026} obtained uniform bounds in
terms of rank and exponent by a probabilistic group-algebra argument.
We obtain the following bound in terms of order and exponent.

\begin{theorem}\label{thm:intro-uniform}
For every finite abelian group $G$,
\begin{equation}\label{eq:main-uniform}
 \D(G)<4.86\,\expn(G)\frac{\log|G|}{\log\expn(G)}.
\end{equation}
\end{theorem}

When
$|G|$ is large compared with $\exp(G)$, it improves \eqref{eq:classical} by a factor of order
$\log\exp(G)$.
The scale in \eqref{eq:main-uniform} is asymptotically attained by
homocyclic $p$-groups as their exponents tend to infinity: if $G=C_n^r$
with $n$ a prime power, then
$\expn(G)\log|G|/\log\expn(G)=rn$ and $\D(G)=r(n-1)+1$.
Our next results describe this scale for almost all finite abelian groups.

We count groups up to isomorphism and order them by their orders. For
$X\ge2$, put
\[
 \mathcal G(X)=\{[G]:1<|G|\le X\},\qquad A(X)=|\mathcal G(X)|,
\]
where $[G]$ denotes the isomorphism class of $G$.
A property holds for almost all finite abelian groups if the proportion
of isomorphism classes in $\mathcal G(X)$ satisfying it tends to $1$ as $X\to\infty$.
The notation $f(G)\sim g(G)$ for almost all $G$ means that, for every
$\varepsilon>0$, the inequality $|f(G)/g(G)-1|<\varepsilon$ holds for
almost all $G$. As usual, $f\asymp g$ means that $f/g$ is bounded above
and below by positive constants, and $f\sim g$ means that $f/g\to1$.

\begin{theorem}\label{thm:intro-global-scale}
As $X\to\infty$,
\begin{equation}\label{eq:main-typical}
 \frac1{A(X)}\sum_{\substack{[G]\\1<|G|\le X}}
 \left|\frac{\expn(G)\log|G|}{\D(G)\log\expn(G)}-1\right|
 \asymp\frac1{\log X}.
\end{equation}
In particular, for almost all finite abelian groups,
\[
 \D(G)\sim\expn(G)\frac{\log|G|}{\log\expn(G)}.
\]
\end{theorem}

Therefore the order--exponent expression in \Cref{thm:intro-uniform} gives
the asymptotic value of $\D(G)$ in this counting model. The one-sided
bound $\D(G)\le\expn(G)\log|G|/\log\expn(G)$ also holds for almost
all groups; see \Cref{prop:almost-all-U-envelope}. The next theorem
compares $\D(G)$ directly with the standard lower bound $\Ds(G)$.

\begin{theorem}\label{thm:intro-mean-ratio}
For every $\delta>0$, as $X\to\infty$,
\begin{equation}\label{eq:main-ratio}
 \frac1{A(X)}\sum_{\substack{[G]\\1<|G|\le X}}
 \frac{\D(G)}{\Ds(G)}
 =1+O_\delta\!\left(X^{-1/2+\delta}\right).
\end{equation}
In particular, for almost all finite abelian groups,
\[
 \D(G)\sim\Ds(G).
\]
\end{theorem}

For the additive difference, we show that for every function
$L:[1,\infty)\to\mathbb R$ with $L(x)\to\infty$, almost all groups
satisfy $\D(G)-\Ds(G)\le L(\Ds(G))$; see
\Cref{prop:intrinsic-additive-upper-tail}. For every fixed $C\ge0$,
however, the groups with $\D(G)-\Ds(G)>C$ have positive lower density
by \Cref{prop:positive-density-defect}.

Starting from Olson's group-ring method \cite{Olson1969,Smertnig2010},
we reduce the proof of \Cref{thm:intro-uniform} to an entropy estimate;
see also Wang \cite[Section~6]{Wang2026} for a related approach.
A weighted comparison of the logarithmic orders and exponents of the
Sylow subgroups, together with an estimate for the Chebyshev function,
yields the bound. For the mean and density results, we separate the
largest cyclic factor from the rest of the group. We then count groups
by the order of the remaining summand and combine
\Cref{thm:intro-uniform} with classical bounds for $\D(G)$.

In \Cref{sec:algebraic-preparation} we give the group-ring criterion and
the required binomial estimate. We prove \Cref{thm:intro-uniform} in
\Cref{sec:order-exponent}. In \Cref{sec:order-density} we first prove
\Cref{thm:intro-mean-ratio}, then use it to prove
\Cref{thm:intro-global-scale}, and finally study the additive difference.

\section{Group rings and the Davenport constant}\label{sec:algebraic-preparation}
	
	We recall the classical group-ring method introduced by Olson and derive some consequences for the Davenport constants.
%The underlying annihilation mechanism is closely related to Wang's argument~\cite[Section~6]{Wang2026}.
	
	Let $G$ be a finite abelian group. Let $S=g_1\bdot\dots\bdot g_\ell$ be a sequence over $G$. We call $S$ zero-sum free if $\sum_{i\in I}g_i\ne0$
	for every nonempty set of positions $I\subseteq\{1,\ldots,\ell\}$.
	Let $K[G]$ be the group algebra over a field $K$, with basis $\{X^g:g\in G\}$ and multiplication $X^gX^h=X^{g+h}$. For background see \cite[Section~5.5]{GeroldingerHalterKoch2006} and
\cite{Smertnig2010}.
	
	\begin{lemma}\label{lem:scalar-nonvanishing}
		Let $S=g_1\bdot\dots\bdot g_N$ be a zero-sum free sequence over $G$. If $\alpha_1,\ldots,\alpha_N\in K^\times$, then
		\[
		\prod_{i=1}^{N}(\alpha_i-X^{g_i})\ne0
		\qquad\text{in }K[G].
		\]
	\end{lemma}
\begin{proof}
The coefficient of $X^0$ in the product equals
$\alpha_1\cdots\alpha_N$, since no nonempty subsequence of $S$
has sum zero. This coefficient is nonzero.
\end{proof}
	
	Let $p$ be prime and
	\[
	P\cong C_{p^{a_1}}\oplus\cdots\oplus C_{p^{a_s}}.
	\]
	Let $K$ be a field of characteristic $p$.
	The augmentation map is
	\[
	\varepsilon:K[P]\longrightarrow K,
	\qquad
	\varepsilon\left(\sum_{g\in P}c_gX^g\right)=\sum_{g\in P}c_g,
	\]
	and its kernel
	\[
	J(P)=\ker\varepsilon
	\]
	is the augmentation ideal. Clearly, there is a $K$-algebra isomorphism
	\[
	K[P]\cong K[T_1,\ldots,T_s]/(T_1^{p^{a_1}},\ldots,T_s^{p^{a_s}}),
	\]
	under which $J(P)$ corresponds to $(T_1,\ldots,T_s)$.
	
	For a $p$-group $P\cong C_{p^{a_1}}\oplus\cdots\oplus C_{p^{a_s}}$,
Olson~\cite{Olson1969} proved
	\begin{equation}\label{eq:olson-formula}
		\D(P)=\Ds(P)=1+\sum_{i=1}^{s}(p^{a_i}-1).
	\end{equation}
	The following nilpotence statement is the group-algebra ingredient
of Olson's argument~\cite{Olson1969}; see also
\cite[Section~5.5]{GeroldingerHalterKoch2006}.
	
	\begin{lemma}\label{lem:nilpotence-index}
		Let $p$ be prime and
		\[
		P\cong C_{p^{a_1}}\oplus\cdots\oplus C_{p^{a_s}}.
		\]
		Then we have
		\[
		J(P)^{\D(P)}=0,
		\qquad
		J(P)^{\D(P)-1}\ne0.
		\]
	\end{lemma}

\begin{proof}
In the polynomial quotient above, a monomial
$T_1^{u_1}\cdots T_s^{u_s}$ is nonzero if and only if
$0\le u_i<p^{a_i}$ for all $i$. Hence the largest total degree
of a nonzero monomial is $\sum_i(p^{a_i}-1)=\D(P)-1$,
by \eqref{eq:olson-formula}. This proves both assertions.
\end{proof}

	Let $H$ be a finite abelian group of exponent $e$, with $p\nmid e$. There is an integer $f\ge1$ such that $e\mid p^f-1$. Take $K=\F_{p^f}$. Then $K^\times$ contains the cyclic group
	\[
	\mu_e=\{z\in K^\times:z^e=1\}
	\]
	of order $e$. Let
	\[
	\widehat H=\Hom(H,\mu_e).
	\]
The character-orthogonality relation (see \cite[Section~6.7]{Apostol1976}
and \cite[Proposition~5.5.2]{GeroldingerHalterKoch2006}) gives
\[
 \sum_{\chi\in\widehat H}\chi(h)=
 \begin{cases}
 |H|,&h=0,\\
 0,&h\ne0,
 \end{cases}
\]
for every $h\in H$. Moreover, $|H|$ is invertible in $K$.

Let $G=P\oplus H$. For $\chi\in\widehat H$, define the $K$-algebra map
	\begin{equation}\label{eq:Phi-character}
		\Phi_\chi:K[G]\longrightarrow K[P],
		\qquad
		\Phi_\chi(X^{(a,b)})=\chi(b)X^a.
	\end{equation}
	
	\begin{lemma}\label{lem:hall-detection}
		If $F\in K[G]$ and $\Phi_\chi(F)=0$ for every $\chi\in\widehat H$, then $F=0$.
	\end{lemma}
	
	\begin{proof}
		Assume that
		\[
		F=\sum_{a\in P}\sum_{b\in H}c_{a,b}X^{(a,b)}.
		\]
		For fixed $a$, the coefficient of $X^a$ in $\Phi_\chi(F)$ is
		\[
		\sum_{b\in H}c_{a,b}\chi(b),
		\]
		which vanishes for every $\chi$. Fix $b_0\in H$, multiply by $\chi(-b_0)$, and sum over all characters. By character orthogonality,
		\begin{align*}
			0
			&=\sum_{\chi\in\widehat H}\chi(-b_0)
			\sum_{b\in H}c_{a,b}\chi(b)\\
			&=\sum_{b\in H}c_{a,b}
			\sum_{\chi\in\widehat H}\chi(b-b_0)
			=|H|c_{a,b_0}.
		\end{align*}
		Since $|H|$ is invertible in $K$, one has $c_{a,b_0}=0$. This holds for every $a$ and $b_0$, so $F=0$.
	\end{proof}
	
	The corresponding characteristic-zero statement appears
in \cite[Proposition~4.0.2]{Dimitrov2007}; for the modular decomposition,
see \cite[Lemma~6.7]{Wang2026}. Here the Hall subgroup is diagonalized,
while the Sylow subgroup remains in characteristic $p$, where its
augmentation ideal is nilpotent.
	
	We use the convention $\binom Nj=0$ when $j>N$.
	
	\begin{proposition}\label{prop:finite-sf}
		Let $G=P\oplus H$ be a finite abelian group, where $|H|>1$, $P$ is an abelian $p$-group with $p\nmid|H|$.
		If a positive integer $N$ satisfies
		\begin{equation}\label{eq:finite-criterion}
			|H|\sum_{j=0}^{\D(P)-1}\binom Nj(\expn(H)-1)^{N-j}<\expn(H)^N,
		\end{equation}
		then $\D(G)\le N$.
	\end{proposition}
	
	\begin{proof}
		For simplicity, we denote
		\[
		d=\D(P)-1,
		\qquad e=\expn(H).
		\]
		Suppose, toward a contradiction, that $N$ satisfies \eqref{eq:finite-criterion}, and let
		\[
		S=g_1\bdot\ldots\bdot g_N,
		\qquad g_i=(x_i,y_i)\in P\oplus H,
		\]
		be zero-sum free. As $p\nmid e$, there is an integer $f\ge1$ such that $e\mid p^f-1$. Take $K=\F_{p^f}$ and $ \mu_e=\{z\in K^\times:z^e=1\}$.
		
		For $\alpha=(\alpha_1,\ldots,\alpha_N)\in\mu_e^N$, let
		\[
		F_\alpha=\prod_{i=1}^{N}(\alpha_i-X^{g_i})\in K[G].
		\]
		Since every $\alpha_i$ is nonzero, \Cref{lem:scalar-nonvanishing} shows
		\begin{equation}\label{eq:Falpha-nonzero}
			F_\alpha\ne0
			\qquad(\alpha\in\mu_e^N).
		\end{equation}
		Fix $\chi\in\widehat H$ and note that, by \eqref{eq:Phi-character},
		\[
		\Phi_\chi(F_\alpha)
		=\prod_{i=1}^{N}\bigl(\alpha_i-\chi(y_i)X^{x_i}\bigr).
		\]
		Whenever $\alpha_i=\chi(y_i)$, the corresponding factor is
		\[
		\alpha_i(1-X^{x_i})\in J(P).
		\]
		By \Cref{lem:nilpotence-index}, $J(P)^{d+1}=0$. Consequently,
		\begin{equation}\label{eq:one-component-annihilation}
			\#\{i:\alpha_i=\chi(y_i)\}\ge d+1
			\quad\Longrightarrow\quad
			\Phi_\chi(F_\alpha)=0.
		\end{equation}

		For a fixed character $\chi$, the number of scalar vectors with exactly $j$ matching coordinates is
		\[
		\binom Nj(e-1)^{N-j}.
		\]
		Indeed, the $j$ matching positions can be chosen in $\binom Nj$ ways; their values are forced to be $\chi(y_i)$, and every remaining coordinate has $e-1$ choices in $\mu_e$. Thus the number of vectors for which \eqref{eq:one-component-annihilation} is not guaranteed, namely the number with at most $d$ matches, is
		\begin{equation}\label{eq:bad-vectors}
			B_N(d,e)=\sum_{j=0}^{d}\binom Nj(e-1)^{N-j}.
		\end{equation}
		Call these vectors bad for $\chi$.
		
		There are $|\widehat H|=|H|$ characters. By the union bound, at most $|H|B_N(d,e)$ scalar vectors are bad for at least one character. The strict inequality \eqref{eq:finite-criterion} says that this number is smaller than
		\[
		e^N=|\mu_e^N|.
		\]
		Hence there is a vector $\alpha\in\mu_e^N$ which is good for every character. For this vector, \eqref{eq:one-component-annihilation} gives
		\[
		\Phi_\chi(F_\alpha)=0
		\qquad(\chi\in\widehat H).
		\]
		It follows from \Cref{lem:hall-detection} that $F_\alpha=0$, contradicting \eqref{eq:Falpha-nonzero}. Therefore no zero-sum-free sequence of length $N$ exists, and $\D(G)\le N$.
	\end{proof}

	Let $X$ be a binomial random variable with parameters $N$ and $1/\expn(H)$, which we denote by $X\sim\operatorname{Bin}(N,1/\expn(H))$. Then
	\begin{equation}\label{eq:binomial-identity}
		\Pr(X\le d)
		=\expn(H)^{-N}\sum_{j=0}^{d}\binom Nj(\expn(H)-1)^{N-j}.
	\end{equation}
	Therefore, \eqref{eq:finite-criterion} is equivalent to
	\begin{equation}\label{eq:probability-criterion}
		|H|\Pr\bigl(\operatorname{Bin}(N,1/\expn(H))\le \D(P)-1\bigr)<1.
	\end{equation}

	For $0<u,v<1$, define the binary relative entropy, or binary Kullback--Leibler divergence, by
	\begin{equation}
		\KL(u\Vert v)
		=u\log\frac uv+(1-u)\log\frac{1-u}{1-v},
		\label{eq:binary-relative-entropy}
	\end{equation}
	with the endpoint convention
	\begin{equation}
		\KL(0\Vert v)=\log\frac1{1-v}.
		\label{eq:binary-relative-entropy-endpoint}
	\end{equation}
	The terminology goes back to Kullback and Leibler~\cite{KullbackLeibler};
see also~\cite[Chapter~2]{CoverThomas}. The estimate below is the
standard Chernoff bound, related to classical work of
Hoeffding~\cite{Hoeffding}.

	\begin{lemma}\label{lem:KL-tail}
		Let $X\sim\operatorname{Bin}(N,v)$ with $0<v<1$, and let $0\le u<v$. Then
		\begin{equation}\label{eq:KL-tail}
			\Pr(X\le uN)\le\exp\{-N\KL(u\|v)\}.
		\end{equation}
	\end{lemma}
	\begin{proof}
		If $u=0$, then
		\[
		\Pr(X\le0)=(1-v)^N=\exp\{-N\KL(0\|v)\}.
		\]
		Assume $0<u<v$. For $t<0$, the event $X\le uN$ implies $e^{tX}\ge e^{tuN}$, so Markov's inequality gives
		\[
		\Pr(X\le uN)
		\le e^{-tuN}\mathbb E(e^{tX})
		=e^{-tuN}(1-v+ve^t)^N.
		\]
		Choose
		\[
		e^t=\frac{u(1-v)}{v(1-u)}<1.
		\]
		A direct substitution yields
		\[
		-tu+\log(1-v+ve^t)=-\KL(u\|v),
		\]
		which proves \eqref{eq:KL-tail}.
	\end{proof}
	
	\begin{lemma}\label{lem:one-sided-KL}
		If $0<u<v<1$, then
		\begin{equation}\label{eq:one-sided-KL}
			\KL(u\|v)\ge v-u+u\log\frac uv.
		\end{equation}
	\end{lemma}
	
	\begin{proof}
		It is enough to prove
		\[
		(1-u)\log\frac{1-u}{1-v}\ge v-u.
		\]
		Take $x=(1-u)/(1-v)>1$. The elementary inequality $\log x\ge1-1/x$ gives
		\[
		(1-u)\log x
		\ge(1-u)\left(1-\frac{1-v}{1-u}\right)
		=v-u,
		\]
		as desired.
	\end{proof}

\section{An upper bound for the Davenport constant}\label{sec:order-exponent}
	
	In this section, we prove \Cref{thm:intro-uniform}. Write
$\psi(x)=\sum_{p^a\le x}\log p$ for the second Chebyshev function.
We use the explicit inequality of Rosser and
Schoenfeld~\cite[Theorem~12]{RosserSchoenfeld1962}:
\begin{equation}\label{eq:Rosser-Schoenfeld}
\psi(x)<1.03883x\qquad(x>0).
\end{equation}
	
	\begin{lemma}\label{lem:local-D-cost}
		Let
		\[
		P\cong C_{p^{b_1}}\oplus\cdots\oplus C_{p^{b_t}},
		\qquad b_i\le a,
		\qquad q=p^a.
		\]
		Then
		\begin{equation}\label{eq:local-D-cost}
			\D(P)-1\le\frac{\log|P|}{\log q}(q-1).
		\end{equation}
	\end{lemma}
	
	\begin{proof}
		By the convexity of $p^x-1$ on $[0,a]$, we have
		\[
		p^b-1\le\frac ba(p^a-1)\qquad(0\le b\le a).
		\]
		Using \eqref{eq:olson-formula}, we have
		\[
		\D(P)-1
		=\sum_i(p^{b_i}-1)
		\le\frac{b_1+\cdots+b_t}{a}(q-1)
		=\frac{\log|P|}{\log q}(q-1),
		\]
		as desired.
	\end{proof}
	
	For $C>1$ and $0<y<C$, define
	\begin{equation}\label{eq:JC-def}
		J_C(y)=C-y-y\log\frac Cy,
	\end{equation}
	and extend continuously by $J_C(0)=C$. Since
	\[
	J_C'(y)=-\log\frac Cy<0\qquad(0<y<C),
	\]
	the function $J_C$ is strictly decreasing on $(0,C)$.
	
	\begin{proposition}\label{prop:distributional-order-exponent}
		Let $c_0=1.03883$. Suppose that $C>1$ and that there is a number $b\in(0,C)$ such that
		\begin{equation}\label{eq:distributional-hypotheses}
			J_C(b)>c_0,
			\qquad
			\int_0^b\left(1-\frac{c_0}{J_C(y)}\right)\,dy>1.
		\end{equation}
		Then every nontrivial finite abelian group $G$ satisfies
		\begin{equation}\label{eq:distributional-rounded}
			\D(G)\le\lceil C\,\expn(G)\frac{\log |G|}{\log\expn(G)}\rceil.
		\end{equation}
	\end{proposition}
	
	\begin{proof}
		If $G$ is a $p$-group, \eqref{eq:olson-formula} together with \Cref{lem:local-D-cost} shows
		\[
		\D(G)\le1+(\expn(G)-1)\frac{\log|G|}{\log\expn(G)}
\le\expn(G)\frac{\log|G|}{\log\expn(G)}
<C\expn(G)\frac{\log|G|}{\log\expn(G)}.
		\]
		We may therefore assume that at least two Sylow subgroups of $G$ are nontrivial.
		
		Let
		\[
		G=\bigoplus_{p\mid|G|}G_p\]
		be the decomposition of $G$ into Sylow subgroups.
		Note that
		\[
		\sum_{p} \log\expn(G_p) =\log\expn(G),
		\qquad
		\sum_p\log|G_p|=\log |G|.
		\]
		Put
		\begin{equation}\label{eq:distributional-weights}
			w_p=\frac{\log\expn(G_p)}{\log\expn(G)},
			\qquad
			y_p=\frac{\log|G_p|}{\log\expn(G_p)} \frac{\log \expn(G)}{\log|G|}.
		\end{equation}
		Then $w_p>0$, $\sum_pw_p=1$, and
		\begin{equation}\label{eq:distributional-mean-one}
			\sum_pw_py_p=1.
		\end{equation}
		Thus the finite collection of values $y_p$, with masses $w_p$, is a probability distribution of mean one.
		
		We first show that there is a prime $p$ such that
		\begin{equation}\label{eq:good-distributional-p}
			y_p\le b\qquad \text{and}\qquad  \expn(G_p)J_C(y_p)>\log\expn(G).
		\end{equation}
		Assume to the contrary that \eqref{eq:good-distributional-p} does not hold for any prime $p$. Let
		\[
		F(a)=\sum_{y_p\le a}w_p
		\qquad(a\ge0)
		\]
		be the distribution function. Fix $0<a\le b$. Since $J_C$ is decreasing and positive on $[0,b]$, every prime with $y_p\le a$ satisfies
		\[
		\expn(G_p)\le\frac{\log\expn(G)}{J_C(y_p)}\le\frac{\log\expn(G)}{J_C(a)}.
		\]
		The numbers $\expn(G_p)$ are pairwise coprime prime powers. Hence their product divides
		\[
		\operatorname{lcm}\!\left(1,2,\ldots,
		\left\lfloor\frac{\log\expn(G)}{J_C(a)}\right\rfloor\right),
		\]
		whenever the set is nonempty. Consequently, using
		\[
		\psi(x)=\log\operatorname{lcm}(1,2,\ldots,\lfloor x\rfloor)
		\]
		and \eqref{eq:Rosser-Schoenfeld}, we have
		\begin{align}
			F(a)
			&=\frac1{\log\expn(G)}\sum_{y_p\le a}\log\expn(G_p)\notag\\
			&\le\frac1{\log\expn(G)}\psi\!\left(\frac{\log\expn(G)}{J_C(a)}\right)
			<\frac{c_0}{J_C(a)}.
			\label{eq:distributional-cdf}
		\end{align}
		
		Let $Y$ be the random variable taking the value $y_p$ with probability $w_p$. By \eqref{eq:distributional-mean-one}, we have
		\[
		1=\mathbb E Y
		=\int_0^\infty\Pr(Y>a)\,da
		=\int_0^\infty(1-F(a))\,da.
		\]
		Using \eqref{eq:distributional-cdf} on $[0,b]$ shows
		\[
		1>\int_0^b\left(1-\frac{c_0}{J_C(a)}\right)\,da,
		\]
		which contradicts our assumption \eqref{eq:distributional-hypotheses}. This proves \eqref{eq:good-distributional-p}.

We now apply the binomial criterion to this Sylow subgroup.
Fix such a prime $p$, and denote
		\[
		P=G_p,\quad H=\bigoplus_{\ell\ne p}G_\ell,
		\quad e=\expn(H)=\frac {\expn(G)}{\expn(G_p)},
		\quad d=\D(P)-1.
		\]
		Since there are at least two nontrivial Sylow subgroups, $H$ is nontrivial. Let
		\[
		N=\lceil C\expn(G)\frac{\log |G|}{\log\expn(G)}\rceil,
		\qquad c=\frac N{\expn(G)},
		\qquad \delta=\frac d{\expn(G_p)}.
		\]
		By \Cref{lem:local-D-cost},
		\begin{equation}\label{deltale}
			\delta\le\frac{\log|G_p|}{\log\expn(G_p)} \frac{\expn(G_p)-1}{\expn(G_p)}<\frac{\log|G_p|}{\log\expn(G_p)} =\frac{\log|G|}{\log\expn(G)}  y_p.
		\end{equation}
		By the choice of $p$, $y_p\le b<C$, and therefore
		\begin{equation}\label{constantcle}
			0<\delta<\frac{\log|G|}{\log\expn(G)} y_p<C\frac{\log|G|}{\log\expn(G)}\le c.
		\end{equation}
		In particular $d/N<1/e$. We now apply \Cref{lem:KL-tail,lem:one-sided-KL}. Let
		\[
		I(c,\delta)=c-\delta-\delta\log\frac c\delta.
		\]
		For $0<\delta<c$,
		\[
		\frac{\partial I}{\partial c}=1-\frac\delta c>0,
		\qquad
		\frac{\partial I}{\partial\delta}=-\log\frac c\delta<0,
		\]
		so $I$ is increasing in $c$ and decreasing in $\delta$. Therefore, by \Cref{lem:one-sided-KL} together with \eqref{deltale} and \eqref{constantcle}, we have
		\begin{align*}
			N\KL\!\left(\frac dN\middle\|\frac1e\right)
			&\ge \expn(G_p)\left(c-\delta-\delta\log\frac c\delta\right)\\
			&>\expn(G_p)\frac{\log|G|}{\log\expn(G)}\left(C-y_p-y_p\log\frac C{y_p}\right)\\
			&=\expn(G_p)\frac{\log|G|}{\log\expn(G)} J_C(y_p)\\
			&>\frac{\log|G|}{\log\expn(G)} \log\expn(G)=\log|G|\ge\log|H|.
		\end{align*}
		It follows from \Cref{lem:KL-tail} and the above discussion that
		\[
		|H|\Pr\!\left(\operatorname{Bin}(N,1/e)\le d\right)<1.
		\]
		By \eqref{eq:probability-criterion} and \Cref{prop:finite-sf}, we have $\D(G)\le N$ and \eqref{eq:distributional-rounded} follows.
	\end{proof}
	
	\begin{proof}[Proof of \Cref{thm:intro-uniform}]
Let
\[
 C_1=\frac{971}{200}=4.855,
 \qquad b=\frac{51}{25}=2.04.
\]
A direct calculation gives
\begin{equation}\label{eq:4855-certificate}
 J_{C_1}(b)>1.0461>1.03883,
 \qquad
 \int_0^b\left(1-\frac{1.03883}{J_{C_1}(y)}\right)dy
 >1.0007>1.
\end{equation}
Hence \Cref{prop:distributional-order-exponent} gives
\begin{equation}\label{eq:4855-rounded}
 \D(G)\le\left\lceil4.855\,
 \expn(G)\frac{\log|G|}{\log\expn(G)}\right\rceil.
\end{equation}
Put $U=\expn(G)\log|G|/\log\expn(G)$. If $U\ge200$, then
\[
 \left\lceil4.855U\right\rceil<4.855U+1\le4.86U.
\]
Suppose that $U<200$. Write $n=\expn(G)$ and
$\tau=\log|G|/\log n\ge1$, so that $U=n\tau$. The classical bound
\eqref{eq:classical} gives
\[
 \frac{\D(G)}U\le\frac1\tau+
 \left(1-\frac1\tau\right)\log n.
\]
If $n\le128$, the right-hand side is at most
$\max\{1,7\log2\}<4.86$. If $n\ge129$, then $n\le U<200$ and
$\tau<200/129$. Since $\log n<\log200<6$,
\[
 \frac{\D(G)}U
 <1+\left(1-\frac{129}{200}\right)6
 =3.13<4.86.
\]
This proves \eqref{eq:main-uniform}.
\end{proof}

\section{Statistical behavior}\label{sec:order-density}

%The proofs in this section use \Cref{thm:intro-uniform}, classical bounds for $\D(G)$, and counting estimates for $|G|/\expn(G)$.
In this section, we prove \Cref{thm:intro-global-scale} and \Cref{thm:intro-mean-ratio}.

%We first remove the largest cyclic invariant factor and count the remaining group. The trivial group is used only as an auxiliary core, with exponent $1$ and rank $0$; all counts involving $A(X)$ exclude it.

%\subsection{Counting by coexponent}

Let $a(m)$ be the number of isomorphism classes of finite abelian groups of order $m$, and therefore
	\[
	A(X)=\sum_{2\le m\le X}a(m)=|\mathcal G(X)|.
	\]
	As in the introduction, $[G]$ denotes the isomorphism class of $G$.
The classical asymptotic
\begin{equation}\label{eq:A-linear}
 A(X)\sim\mathfrak A X,
 \qquad
 \mathfrak A=\prod_{j=2}^{\infty}\zeta(j)<\infty,
\end{equation}
goes back to Erd\H{o}s and Szekeres \cite{ErdosSzekeres1935}; see also
Heath-Brown \cite[Section~1]{HeathBrown1991} for stronger expansions.

For a nontrivial finite abelian group $G$, we denote
	\begin{equation}\label{eq:coexponent}
		Q(G)=\frac{|G|}{\expn(G)}.
	\end{equation}
	
	\begin{lemma}\label{lem:coexponent-parameterization}
		For every nontrivial finite abelian group $G$, there is a unique pair $(H,t)$, where $H$ is a finite abelian group and $t\ge1$ is an integer, such that
		\begin{equation}\label{eq:core-param}
			G\cong H\oplus C_{t\expn(H)}
		\end{equation}
		and
		\begin{equation}\label{eq:core-order}
			|H|=Q(G).
		\end{equation}
	\end{lemma}
	
	\begin{proof}
		Write $G$ in invariant-factor form, and put
		\[
		H=C_{n_1}\oplus\cdots\oplus C_{n_{r-1}},
		\]
		with $H$ trivial if $r=1$. Then $\expn(H)=n_{r-1}$ when $r\ge2$, and $n_{r-1}\mid n_r$. Thus
		\[
		n_r=t\expn(H)
		\]
		for a unique integer $t\ge1$, and \eqref{eq:core-param} holds. Also,
		\[
		|H|=\frac{|G|}{n_r}=Q(G).
		\]
		Conversely, if $H$ is trivial, then $H\oplus C_t=C_t$, and the pair is determined by the single cyclic invariant factor. Suppose now that $H$ is nontrivial. If $(H,t)$ is given and
		\[
		H\cong C_{d_1}\oplus\cdots\oplus C_{d_s},
		\qquad d_1\mid\cdots\mid d_s=\expn(H),
		\]
		then
		\[
		H\oplus C_{t\expn(H)}
		\]
		is already in invariant-factor form. The invariant-factor decomposition is unique, so the pair $(H,t)$ is unique.
	\end{proof}
	
	For $q\ge1$, define
	\begin{equation}\label{eq:beta-def} \beta(q)=\sum_{\substack{[H]\\|H|=q}}\frac1{\expn(H)},
	\end{equation}
	where the trivial group has exponent $1$.
	
	\begin{lemma}\label{lem:B-convergence}
		For every $\sigma>0$,
		\begin{equation}\label{eq:B-sigma}
			B(\sigma)=\sum_{q\ge1}\frac{\beta(q)}{q^\sigma}<\infty.
		\end{equation}
	\end{lemma}
	
	\begin{proof}
		The function $\beta$ is multiplicative. Indeed, if $(q_1,q_2)=1$, every group of order $q_1q_2$ decomposes uniquely as $H_1\oplus H_2$ with $|H_i|=q_i$, and
		\[
		\expn(H_1\oplus H_2)=\expn(H_1)\expn(H_2).
		\]
		A finite abelian $p$-group corresponds to a partition $\lambda=(\lambda_1\ge\lambda_2\ge\cdots)$ with order $p^{|\lambda|}$ and  exponent $p^{\lambda_1}$. Therefore
		\begin{equation}\label{eq:B-Euler}
			B(\sigma)
			=\prod_p\left(1+\sum_{\lambda\ne\varnothing}p^{-\lambda_1-\sigma|\lambda|}\right).
		\end{equation}
		For each fixed prime $p$, the local sum converges. In fact, it is bounded by
		\[
		p^{-1}\sum_{k\ge1}p(k)p^{-\sigma k},
		\]
		where $p(k)$ denotes the partition number. The partition generating function
		\[
		\sum_{k\ge0}p(k)x^k=\prod_{j\ge1}(1-x^j)^{-1}
		\]
		converges for $0<x<1$: indeed,
		\[
		\sum_{j\ge1}-\log(1-x^j)
		\le\sum_{j\ge1}\frac{x^j}{1-x^j}
		\le\frac{x}{(1-x)^2}<\infty.
		\]
		Taking $x=p^{-\sigma}$ proves convergence of every local factor.
		
		For uniform control at large primes, use the elementary bound $p(k)\le2^{k-1}$, obtained by mapping a partition of $k$ to a composition of $k$. For all sufficiently large $p$, one has $2p^{-\sigma}<1$, and
		\begin{align*}
			\sum_{\lambda\ne\varnothing}p^{-\lambda_1-\sigma|\lambda|}
			&\le p^{-1}\sum_{k\ge1}p(k)p^{-\sigma k}\\
			&\le p^{-1}\sum_{k\ge1}2^{k-1}p^{-\sigma k}
			=\frac{p^{-1-\sigma}}{1-2p^{-\sigma}}
			=O_\sigma(p^{-1-\sigma}).
		\end{align*}
		The finitely many remaining local factors converge individually, while
		\[
		\sum_pp^{-1-\sigma}<\infty.
		\]
		Thus the Euler product \eqref{eq:B-Euler} converges.
	\end{proof}
	
	Denote
	\begin{equation}\label{eq:T-def}
		T(Y)=\sum_{q>Y}\frac{\beta(q)}q.
	\end{equation}
	By \Cref{lem:B-convergence} with $\sigma=1$, this is finite for every $Y\ge1$ and tends to zero as $Y\to\infty$.
	
	\begin{proposition}\label{prop:coexponent-tails}
		For all $X,Y\ge1$,
		\begin{equation}\label{eq:exact-coexponent-tail}
			\#\{[G]:1<|G|\le X,\ Q(G)>Y\}
			\le X\sum_{q>Y}\frac{\beta(q)}q
			=XT(Y).
		\end{equation}
		For every $0<\eta<1$,
		\begin{equation}\label{eq:polynomial-coexponent-tail}
			\#\{[G]:1<|G|\le X,\ Q(G)>Y\}
			\le B(\eta)XY^{-1+\eta}.
		\end{equation}
		In particular, for every $\varepsilon>0$ there is $Y$ such that
		\[
		\limsup_{X\to\infty}
		\frac{\#\{[G]:1<|G|\le X,\ Q(G)>Y\}}{A(X)}
		<\varepsilon.
		\]
	\end{proposition}
	
	\begin{proof}
		By \Cref{lem:coexponent-parameterization}, for each $q\ge2$,
		\begin{equation}\label{eq:exact-fiber-count}
			\#\{[G]:1<|G|\le X,\ Q(G)=q\}
			=\sum_{\substack{[H]\\|H|=q}}
			\left\lfloor\frac{X}{q\expn(H)}\right\rfloor.
		\end{equation}
		Indeed, the integer $t$ in \eqref{eq:core-param} is constrained by
		\[
		|G|=|H|t\expn(H)=qt\expn(H)\le X.
		\]
		Summing \eqref{eq:exact-fiber-count} over $q>Y$ and replacing each floor by its argument gives \eqref{eq:exact-coexponent-tail}.
		
		For $q>Y$, one has
		\[
		q^{-1}\le Y^{-1+\eta}q^{-\eta}.
		\]
		Hence
		\[
		T(Y)
		=\sum_{q>Y}\frac{\beta(q)}q
		\le Y^{-1+\eta}\sum_{q\ge1}\frac{\beta(q)}{q^\eta}
		=B(\eta)Y^{-1+\eta},
		\]
		which proves \eqref{eq:polynomial-coexponent-tail}. Divide \eqref{eq:exact-coexponent-tail} by $A(X)\sim\mathfrak A X$. Given $\varepsilon>0$, choose $Y$ so large that $T(Y)<\mathfrak A\varepsilon/2$. For all sufficiently large $X$, $A(X)>\mathfrak AX/2$, and hence the proportion of groups with $Q(G)>Y$ is less than $\varepsilon$.
	\end{proof}
	
%	\subsection{The mean Davenport ratio}

The coexponent tail estimate can be transferred to $\D(G)$ through
the following consequence of the bounds of Bhowmik and Schlage-Puchta.
For every nontrivial finite abelian group $G$, put
$\Def(G)=\D(G)-\Ds(G)$.

\begin{lemma}\label{lem:coexponent-transfer}
For every nontrivial finite abelian group $G$,
\begin{equation}\label{eq:universal-defect}
 0\le\Def(G)\le Q(G)-1.
\end{equation}
\end{lemma}

\begin{proof}
Put $n=\expn(G)$ and $Q=Q(G)$. If $n\ge\sqrt{|G|}$,
Bhowmik and Schlage-Puchta \cite[Theorem~1.1]{BhowmikSchlagePuchta2012}
give
\[
 \D(G)\le n+\frac{|G|}{n}-1=n+Q-1.
\]
If $n<\sqrt{|G|}$, the second assertion of the same theorem gives
\[
 \D(G)\le2\sqrt{|G|}-1=2\sqrt{nQ}-1\le n+Q-1.
\]
Thus
\begin{equation}\label{eq:D-nQ}
 \D(G)\le n+Q-1.
\end{equation}
Since the largest invariant factor of $G$ is $n$, one has
$\Ds(G)\ge n$. Subtracting proves \eqref{eq:universal-defect}.
\end{proof}

\begin{proof}[Proof of \Cref{thm:intro-mean-ratio}]

		Let
		\[
		R(G)=\frac{\Def(G)}{\Ds(G)}.
		\]
		Then $R(G)\ge0$ and
		\begin{equation}\label{eq:E-def}
			\sum_{\substack{[G]\\1<|G|\le X}}
			\frac{\D(G)}{\Ds(G)}
			=A(X)+E(X),
			\quad \text{where }
			E(X)=\sum_{\substack{[G]\\1<|G|\le X}}R(G).
		\end{equation}
		
By \Cref{lem:coexponent-parameterization}, write
\[
 G\cong H\oplus C_{te},\qquad q=|H|,\quad e=\expn(H),\quad r=\rank(H).
\]
If $H$ is trivial, then $G$ is cyclic and $R(G)=0$, so we may assume $q>1$.
Since $\Ds(G)\ge te$, \Cref{lem:coexponent-transfer} gives
$R(G)\le(q-1)/(te)$. Also $|G|\le\expn(G)^{\rank(G)}$, so
\Cref{thm:intro-uniform} gives
\[
 \D(G)<4.86\,\expn(G)\rank(G)=4.86\,te(r+1).
\]
Thus $R(G)\le4.86(r+1)$. Taking the geometric mean of these two bounds,
we obtain
\[
 R(G)\le\sqrt{\frac{4.86(q-1)(r+1)}{te}}.
\]
For fixed $H$, the order condition is $t\le X/(qe)$. Using
$\sum_{1\le t\le Z}t^{-1/2}\le2\sqrt Z$, we therefore have
\begin{align*}
 E(X)
 &\le\sqrt{4.86}\sum_{\substack{[H]\\q>1,\ qe\le X}}
 \sqrt{\frac{(q-1)(r+1)}e}
 \sum_{t\le X/(qe)}t^{-1/2}\\*
 &\le2\sqrt{4.86X}\sum_{\substack{[H]\\1<q\le X}}
 \frac{\sqrt{r+1}}e.
\end{align*}
Fix $\delta>0$. Since $r\le\log q/\log2$, we have
$\sqrt{r+1}\ll_\delta q^{\delta/2}$. Hence \Cref{lem:B-convergence}
yields
\[
 E(X)\ll_\delta\sqrt X
 \sum_{\substack{[H]\\1<q\le X}}\frac{q^{\delta/2}}e
 \le X^{1/2+\delta}\sum_{[H]}\frac1{|H|^{\delta/2}\expn(H)}
 =B(\delta/2)X^{1/2+\delta}.
\]
Consequently $E(X)=O_\delta(X^{1/2+\delta})$.
		Together with \eqref{eq:E-def}, this gives
\[
 \sum_{\substack{[G]\\1<|G|\le X}}
 \frac{\D(G)}{\Ds(G)}
 =A(X)+O_\delta(X^{1/2+\delta}).
\]
Finally, \eqref{eq:A-linear} gives $A(X)\asymp X$, and division by
$A(X)$ proves \eqref{eq:main-ratio}.

For the almost-all conclusion, fix $\varepsilon>0$. If
$\D(G)>(1+\varepsilon)\Ds(G)$, then
$R(G)>\varepsilon$. Hence
\[
 \varepsilon\#\{[G]:1<|G|\le X,\
 \D(G)>(1+\varepsilon)\Ds(G)\}\le E(X).
\]
Together with $E(X)=O_\delta(X^{1/2+\delta})$ and
$A(X)\asymp X$, this shows that the exceptional proportion is
$O_{\varepsilon,\delta}(X^{-1/2+\delta})$, and in particular tends to
zero.
\end{proof}

%\subsection{The typical order--exponent scale}

For every finite abelian group, let
	\[
	U(G)=\expn(G)\frac{\log|G|}{\log\expn(G)}.
	\]
	We first compare $U(G)$ with $\Ds(G)$, then use
\Cref{thm:intro-mean-ratio} to replace $\Ds(G)$ by $\D(G)$.
	
	\begin{lemma}\label{lem:Dstar-U}
		For every finite abelian group $G$,
		\begin{equation}\label{eq:Dstar-U}
			\Ds(G)\le U(G).
		\end{equation}
	\end{lemma}
	
	\begin{proof}
		Write
		\[
		G\cong C_{n_1}\oplus\cdots\oplus C_{n_r},
		\qquad 1<n_1\mid\cdots\mid n_r=n=\expn(G),
		\]
		and put $\tau=\log|G|/\log n\ge1$. The function $x\mapsto\log x/(x-1)$ is strictly decreasing on $(1,\infty)$, since its derivative has the sign of
		\[
		\frac{x-1}{x}-\log x<0.
		\]
		Hence, for every $i$,
		\[
		n_i-1\le\frac{n-1}{\log n}\log n_i.
		\]
		Summing gives
		\[
		\Ds(G)
		\le1+\frac{n-1}{\log n}\sum_{i=1}^{r}\log n_i
		=1+(n-1)\tau
		\le n\tau
		=U(G).
		\]
	\end{proof}

	\begin{proposition}\label{prop:U-Dstar-mean}
		As $X\to\infty$,
		\begin{equation}\label{eq:U-Dstar-sharp-mean}
			\frac1{A(X)}
			\sum_{\substack{[G]\\1<|G|\le X}}
			\left(\frac{U(G)}{\Ds(G)}-1\right)
			\asymp\frac1{\log X}.
		\end{equation}
	\end{proposition}
	
	\begin{proof}
		Fix a nontrivial finite abelian group
		\[
		H\cong C_{m_1}\oplus\cdots\oplus C_{m_s},
		\qquad e=\expn(H),
		\qquad q=|H|,
		\]
		and let
		\[
		a=\Ds(H)-1=\sum_i(m_i-1).
		\]
		The groups in the $H$-fiber are $G=H\oplus C_{te}$, with
		\[
		|G|=qet,
		\qquad
		\expn(G)=et,
		\qquad
		\Ds(G)=et+a.
		\]
		Consequently,
		\begin{equation}\label{eq:U-Dstar-fiber}
			\frac{U(G)}{\Ds(G)}
			=\frac{et}{et+a}\left(1+\frac{\log q}{\log(et)}\right),
		\end{equation}
		and \Cref{lem:Dstar-U} implies
		\begin{equation}\label{eq:U-Dstar-pointwise-fiber}
			0\le\frac{U(G)}{\Ds(G)}-1
			\le\frac{\log q}{\log(et)}.
		\end{equation}
		The trivial core $H=0$ produces cyclic groups, for which $U(G)=\Ds(G)$, so it contributes zero and may be omitted.
		
		Let
		\[
		T_H=\left\lfloor\frac{X}{qe}\right\rfloor
		\]
		and split the fibers according as $qe\le X^{1/2}$ or $qe>X^{1/2}$.
		
		If $qe\le X^{1/2}$, then $T_H\ge X^{1/2}-1$. Uniformly for $e\ge2$ and $T\ge X^{1/2}/2$,
		\begin{equation}\label{eq:log-harmonic}
			\sum_{t\le T}\frac1{\log(et)}\ll\frac{T}{\log X}.
		\end{equation}
		Indeed, split at $\sqrt T$: the first part is $O(\sqrt T)$, and on the second part $\log(et)\ge\frac12\log T\gg\log X$. Hence the total contribution of these fibers is
		\[
		\ll\frac{X}{\log X}
		\sum_{[H]}\frac{\log|H|}{|H|\expn(H)}
		\ll\frac{X}{\log X},
		\]
		where the last series converges by \Cref{lem:B-convergence}, since $q^{-1}\log q\ll q^{-1/2}$.
		
		For the remaining fibers, $qe>X^{1/2}$. Since $e\le q$ for nontrivial $H$, one has $q>X^{1/4}$. Also $q\le e^s$, so
		\[
		\frac{\log q}{\log e}\le s=\rank(H).
		\]
		Their total contribution is at most
		\[
		X\sum_{\substack{[H]\\qe>X^{1/2}}}
		\frac{\rank(H)}{qe}
		\ll X\sum_{q>X^{1/4}}\frac{\beta(q)\log q}{q}.
		\]
		For $Z$ sufficiently large, the function $q\mapsto(\log q)/q^{1/2}$ is decreasing, and therefore
		\[
		\sum_{q>Z}\frac{\beta(q)\log q}{q}
		\le B(1/2)\sup_{q>Z}\frac{\log q}{q^{1/2}}
		\ll Z^{-1/2}\log Z.
		\]
		With $Z=X^{1/4}$, the short-fiber contribution is $O(X^{7/8}\log X)=o(X/\log X)$. Since $A(X)\asymp X$, this proves the upper bound in \eqref{eq:U-Dstar-sharp-mean}.
		
		For the reverse bound, take the rank-two groups
		\[
		G_t=C_2\oplus C_{2t}.
		\]
		It is well-known (see \cite[Theorem~3.1]{GaoGeroldinger2006}) that
		\[
		\Ds(G_t)=\D(G_t)=2t+1,
		\]
		whereas
		\[
		U(G_t)=2t\frac{\log(4t)}{\log(2t)}.
		\]
		For $X/8\le t\le X/4$ and $X$ sufficiently large,
		\begin{align*}
			\frac{U(G_t)}{\Ds(G_t)}-1
			&=\frac{2t\log2-\log(2t)}{(2t+1)\log(2t)}
			\gg\frac1{\log X}.
		\end{align*}
		There are $\gg X$ such pairwise nonisomorphic groups, all of order at most $X$. Dividing by $A(X)\asymp X$ gives the matching lower bound.
	\end{proof}
	
	\begin{lemma}\label{lem:pointwise-U-envelope}
		Let $G$ be a finite abelian group. If $Q(G)\le \expn(G)$, then
		\begin{equation}\label{eq:pointwise-U-envelope}
			\D(G)\le U(G).
		\end{equation}
	\end{lemma}
	
	\begin{proof}
		By Bhowmik and Schlage-Puchta \cite[Theorem~1.1]{BhowmikSchlagePuchta2012}, in the range $\expn(G)\ge\sqrt{|G|}$ one has
		\[
		\D(G)\le \expn(G)+Q(G)-1.
		\]
		The assumption $Q(G)\le \expn(G)$ is equivalent to $\expn(G)\ge\sqrt{|G|}$, so this estimate applies. If $Q(G)=1$, then $G$ is cyclic and $\D(G)=\expn(G)=U(G)$. Assume $1<Q(G)\le \expn(G)$. It is easy to see that the function
		\(
		f(x)=\frac{x-1}{\log x}
		\)
		is strictly increasing on $(1,\infty)$. Hence
		\[
		\frac{Q(G)-1}{\log Q(G)}
		\le\frac{\expn(G)-1}{\log \expn(G)}
		<\frac {\expn(G)}{\log \expn(G)},
		\]
		and therefore
		\[
		Q(G)-1\le \expn(G)\frac{\log Q(G)}{\log \expn(G)}.
		\]
		Since $|G|=\expn(G)Q(G)$,
		\[
		\expn(G)+Q(G)-1
		\le \expn(G)+\expn(G)\frac{\log Q(G)}{\log \expn(G)}
		=\expn(G)\frac{\log(\expn(G)Q(G))}{\log \expn(G)}
		=U(G).
		\]
	\end{proof}
	
	\begin{proposition}\label{prop:almost-all-U-envelope}
		For every $0<\delta<1/2$,
		\begin{equation}\label{eq:almost-all-U-envelope-count}
			\#\{[G]:1<|G|\le X,\ \D(G)>U(G)\}
			=O_\delta\!\left(X^{1/2+\delta}\right).
		\end{equation}
		In particular, $\D(G)\le U(G)$ for almost all finite abelian groups.
	\end{proposition}
	
	\begin{proof}
		By \Cref{lem:pointwise-U-envelope}, every group counted on the left-hand side satisfies
		\[
		Q(G)>\expn(G).
		\]
		Use the unique coexponent parameterization
		\[
		G\cong H\oplus C_{t\expn(H)}.
		\]
		Write
		\[
		q=|H|=Q(G),\qquad e=\expn(H).
		\]
		Then
		\[
		\expn(G)=te,\qquad |G|=qte.
		\]
		The inequality $Q(G)>\expn(G)$ gives $t<q/e$, while $|G|\le X$ gives $t\le X/(qe)$. Thus, for fixed $H$, the number of possible integers $t$ is at most
		\[
		\frac1e\min\left\{q,\frac Xq\right\}.
		\]
		Summing over cores $H$ and using the definition of $\beta(q)$,
		\begin{equation}\label{eq:bad-U-split}
			\#\{[G]:1<|G|\le X,\ Q(G)>\expn(G)\}
			\le\sum_{q\ge1}\beta(q)\min\left\{q,\frac Xq\right\}.
		\end{equation}
		Fix $0<\eta<1$ and split at $q=X^{1/2}$. For $q\le X^{1/2}$,
		\[
		q\beta(q)
		=q^{1+\eta}\frac{\beta(q)}{q^\eta}
		\le X^{(1+\eta)/2}\frac{\beta(q)}{q^\eta}.
		\]
		For $q>X^{1/2}$,
		\[
		\frac{X\beta(q)}q
		=Xq^{\eta-1}\frac{\beta(q)}{q^\eta}
		\le X^{(1+\eta)/2}\frac{\beta(q)}{q^\eta}.
		\]
		Therefore \Cref{lem:B-convergence} yields
		\[
		\#\{[G]:1<|G|\le X,\ Q(G)>\expn(G)\}
		\le2B(\eta)X^{(1+\eta)/2}.
		\]
		Choose $\eta=2\delta$. This proves \eqref{eq:almost-all-U-envelope-count}. Since $A(X)\sim\mathfrak A X$, the exceptional proportion tends to zero.
	\end{proof}
	
	\begin{proof}[Proof of \Cref{thm:intro-global-scale}]

		Set
		\[
		A_G=\frac{\D(G)}{\Ds(G)},
		\qquad
		B_G=\frac{U(G)}{\Ds(G)}.
		\]
		By \eqref{eq:Dstar} and \Cref{lem:Dstar-U}, $A_G,B_G\ge1$. Hence
		\begin{align*}
			\left|\frac{U(G)}{\D(G)}-1\right|
			&=\left|\frac{B_G}{A_G}-1\right|\\
			&\le |B_G-A_G|\\
			&\le(A_G-1)+(B_G-1).
		\end{align*}
		The average of $A_G-1$ is $O(X^{-1/4})$ by
\Cref{thm:intro-mean-ratio}, taking $\delta=1/4$, while \Cref{prop:U-Dstar-mean} gives an $O(1/\log X)$ bound for the average of $B_G-1$. This proves the upper bound in \eqref{eq:main-typical}.
		
		For the reverse bound, we consider
		\[
		G_t=C_2\oplus C_{2t},\qquad X/8\le t\le X/4.
		\]
		As in the proof of \Cref{prop:U-Dstar-mean}, we have
		$
		\frac{U(G_t)}{\D(G_t)}-1
		\gg\frac1{\log X}$.
		There are $\gg X$ such pairwise nonisomorphic groups of order at most $X$. Since $A(X)\asymp X$, this gives the matching lower bound.

The almost-all assertion follows from Markov's inequality.
\end{proof}

Combining \Cref{prop:almost-all-U-envelope,thm:intro-global-scale}, for
every fixed $0<c<1$ one has
\[
 c\,U(G)<\D(G)\le U(G)
\]
for almost all finite abelian groups.

%\subsection{The additive difference}

Let $L:[1,\infty)\to\mathbb R$ satisfy $L(x)\to\infty$. Only finitely
many groups have $\Ds(G)$ in a fixed bounded interval. Hence changing $L$
on a bounded interval does not affect the assertion below, and we may
assume that $L$ is bounded below. Define
\begin{equation}\label{eq:Lstar}
 L^*(u)=\inf_{x\ge u}L(x).
\end{equation}
Then $L^*$ is finite and nondecreasing, $L^*(u)\le L(u)$, and
$L^*(u)\to\infty$.
	
	\begin{proposition}\label{prop:intrinsic-additive-upper-tail}
		Let
		\[
		E_L^*(X)=\#\{[G]:1<|G|\le X,\ \Def(G)>L(\Ds(G))\}.
		\]
		Then $E_L^*(X)=o(X)$. In particular,
		\begin{equation}\label{eq:intrinsic-additive-aa}
			\Def(G)\le L(\Ds(G))
		\end{equation}
		for almost all finite abelian groups.
	\end{proposition}
	
	\begin{proof}
		For sufficiently large $X$, put
		\[
		u_X=\lceil X^{1/3}\rceil,
		\qquad
		Y_X=\min\!\left\{\lfloor L^*(u_X)\rfloor+1,\ \lfloor X^{1/3}\rfloor\right\}.
		\]
		Then $Y_X\to\infty$,
		\[
		Y_X-1\le L^*(u_X),
		\qquad
		u_XY_X=O(X^{2/3})=o(X).
		\]
		Suppose that $G$ is counted by $E_L^*(X)$ and $Q(G)\le Y_X$. If $\Ds(G)\ge u_X$, then \Cref{lem:coexponent-transfer} gives
		\[
		\Def(G)\le Q(G)-1
		\le Y_X-1
		\le L^*(u_X)
		\le L(\Ds(G)),
		\]
		a contradiction. Hence every exceptional group with $Q(G)\le Y_X$ satisfies $\Ds(G)<u_X$.
		
		Since $\Ds(G)\ge\expn(G)$,
		\[
		|G|=\expn(G)Q(G)
		\le\Ds(G)Q(G)
		<u_XY_X.
		\]
		Thus every exceptional group either has order below $u_XY_X$ or satisfies $Q(G)>Y_X$. By \Cref{prop:coexponent-tails},
		\[
		E_L^*(X)\le A(u_XY_X)+XT(Y_X).
		\]
		The first term is $o(X)$ by \eqref{eq:A-linear}, and the second is $o(X)$ because $Y_X\to\infty$ and $T(Y)\to0$. Hence $E_L^*(X)=o(X)$. Dividing by $A(X)\sim\mathfrak A X$ proves the assertion.
	\end{proof}
	
%We shall use Liu's lower bound \cite[Theorem~4.3]{Liu2020}; his function $\theta(\ell;p)$ is defined in \cite[Section~2]{Liu2020}.

The following result was proved in \cite{Liu2020}. For the notation $\theta(\ell)$ and $\theta(\ell;p)$, which will be used later, we refer to Section 2 of that paper.

\begin{lemma}[\cite{Liu2020}, Theorem~4.3]\label{LowerDG}
Let \(G = C_n^r \oplus C_{kn}\) be a non-\(p\)-group with \(n, k \ge 2\). Let \(p\) be a prime divisor of \(n\), \(k_1 \in \mathbb{Z}_{\ge 2}\) be a divisor of \(k\), with \(\gcd(p, k_1) = 1\), and \(kn = k_1m\) for some \(m \in \mathbb{Z}_{>0}\). If \(\ell \in \mathbb{Z}_{>0}\) and \(t \in [1, k_1 - 1]\) with \(r \ge t\theta(\ell) \ge 1\), then
\[
\mathsf{D}(G) \ge \mathsf{D}^*(G) + t(p - 1)\ell - tm.
\]
\end{lemma}

\begin{proposition}\label{prop:positive-density-defect}
		For every real constant $C\ge0$,
		\begin{equation}\label{eq:positive-density-defect}
			\liminf_{X\to\infty}
			\frac{\#\{[G]:1<|G|\le X,\ \Def(G)>C\}}{A(X)}
			\ge\frac1{64\mathfrak A\,16^C}>0.
		\end{equation}
	\end{proposition}
	
	\begin{proof}
		Let $s=\lfloor C\rfloor+1$. For every odd integer $k>s$, consider
		\begin{equation}\label{eq:Gk}
			G_k=C_2^{4s}\oplus C_{2k}.
		\end{equation}
		We apply \Cref{LowerDG}, choose
		\[
		n=p=2,
		\qquad k_1=k,
		\qquad m=2,
		\qquad \ell=3,
		\qquad t=s,
		\qquad r=4s.
		\]
		First, $G_k=C_n^r\oplus C_{kn}$ with $n=2$. Since $k>1$ is odd, $G_k$ is not a $p$-group. The prime $p=2$ divides $n$, while $k_1=k$ divides $k$, satisfies $k_1\ge2$, and is relatively prime to $p$. Moreover, $kn=2k=k_1m$. For $p=2$ and $\ell=3$ one has
		\[
		\theta(3;2)=2^3-1-3=4.
		\]
		Consequently,
		\(
		r=4s=s\theta(3;2)=t\theta(\ell;p).
		\)
		Finally, $1\le t=s\le k-1=k_1-1$ because $k>s$. Therefore \Cref{LowerDG} gives
		\begin{equation}\label{eq:Liu-defect}
			\D(G_k)-\Ds(G_k)
			\ge t(p-1)\ell-tm
			=s(2-1)3-2s=s.
		\end{equation}
		Since $s>C$, every such $G_k$ has defect larger than $C$.
		
		The groups $G_k$ are pairwise nonisomorphic because their exponents are the distinct integers $2k$, and
		\begin{equation}\label{eq:Gk-order}
			|G_k|=2^{4s+1}k.
		\end{equation}
		Clearly, the number of odd integers $k>s$ with $|G_k|\le X$ is
		\begin{equation}\label{eq:k-count}
			\frac{X}{2^{4s+2}}+O_C(1).
		\end{equation}
		Using $A(X)\sim\mathfrak A X$ from \eqref{eq:A-linear}, equations \eqref{eq:Liu-defect}--\eqref{eq:k-count} imply
		\[
		\liminf_{X\to\infty}
		\frac{\#\{[G]:1<|G|\le X,\ \Def(G)>C\}}{A(X)}
		\ge\frac1{\mathfrak A2^{4s+2}}.
		\]
		Finally, $s\le C+1$, so
		\[
		2^{4s+2}=4\cdot16^s\le64\cdot16^C.
		\]
		This completes the proof.
	\end{proof}

\medskip
\noindent\textbf{Statement of AI.} The paper was written by the authors.
ChatGPT was used to assist with routine calculations and checks, language
editing, and reorganization of the manuscript.

\medskip
\noindent\textbf{Acknowledgment.} This work was supported by Guangdong
Basic and Applied Basic Research Foundation Grant No.~2024A1515012564 and
the National Natural Science Foundation of China Grant No.~12571578.

\end{document}